\documentclass[12pt]{amsart}
\usepackage{amsmath}
\usepackage{amsfonts}
\usepackage{amssymb}
\usepackage{graphicx}
\theoremstyle{plain}
  \newtheorem{theorem}[subsection]{Theorem}
  
  \newtheorem{proposition}[subsection]{Proposition}
  \newtheorem{lemma}[subsection]{Lemma}

\theoremstyle{remark}
  \newtheorem{remark}[subsection]{Remark}

\theoremstyle{definition}
  \newtheorem{definition}[subsection]{Definition}

\begin{document}

\title[]
{The linear profile decomposition for the fourth order Schr\"odinger equation}%
\author{Jin-Cheng Jiang}
\address{Institute of Mathematics, Academic Sinica, Taipei, Taiwan 11529, R.O.C}
\email{jiang@math.sinica.edu.tw}
\author{Benoit Pausader}
\address{Mathematics Department,
Box 1917,
Brown University,
Providence, RI 02912}
\email{Benoit.Pausader@math.brown.edu}

\author{Shuanglin Shao}
\address{Institute for Mathematics and its applications, University of Minnesota, Minneapolis, MN 55455}
\email{slshao@ima.umn.edu}

\date{\today}
\begin{abstract}
In this paper, we establish the linear profile decomposition for the one dimensional fourth order
Schr\"{o}dinger equation $$
\begin{cases}
iu_t-\mu\Delta u+\Delta^2u=0\;,\;t\in\mathbb{R},\,x\in\mathbb{R},\\
u(0,x)=f(x)\in L^2,
\end{cases}$$where $\mu\ge 0$.  As an application, we establish a dichotomy result on the existence of extremals to the symmetric Schr\"{o}dinger Strichartz inequality.
\end{abstract}

\maketitle
\section{Introduction}
\subsection{Linear profile decomposition}
In this paper, we consider the problem of the linear profile decomposition for the fourth order
 Schr\"odinger equation of the following form with $L^2$ data in one spatial dimension
\begin{equation}\label{eq:4thNLS}
\begin{cases}
iu_t-\mu\Delta u+\Delta^2u=0\;,\;t\in\mathbb{R},\,x\in\mathbb{R},\\
u(0,x)=f(x)\in L^2,
\end{cases}\end{equation}
where $u:\mathbb{R}\times \mathbb{R}\rightarrow \mathbb{C}$ and $\mu\ge 0$\footnote{The case $\mu<0$ is intentionally not included due to lack of a refinement of Strichartz inequality, cf.  the inequality \eqref{eq:refined-strichartz} when $\mu\ge 0$. Moreover, the global Strichartz estimate may not be available in view of the presence of the degenerate critical point for the phase function, see e.g. \cite[Condition (2.1.c)]{Kenig-Ponce-Vega:1991:dispersive-estimates} or  \cite[(10)]{BenArtzi-Koch-Saut:2000:dispersive-estimate-for-4th-NLS}.}. Equation \eqref{eq:4thNLS} is the free form of one dimensional fourth-order nonlinear Schr\"odinger equations that have been introduced by Karpman \cite{Karpman:1996} and Karpman-Shagalov \cite{Karpmam-Shagalov:2000} to take into account the role of ``fourth-order dispersion" in the propagation of intense laser beams in a bulk medium with Kerr nonlinearity.

The main result in this paper, the linear profile decomposition for Equation \eqref{eq:4thNLS}, is motivated by the analogous decompositions in context of wave, Schr\"odinger and Airy equations  \cite{Bahouri-Gerard:1999:profile-wave, Begout-Vargas:2007:profile-schrod-higher-d, Carles-Keraani:2007:profile-schrod-1d, Keraani:2001:profile-schrod-H^1, Merle-Vega:1998:profile-schrod, Shao:2008:linear-profile-Airy-Maximizer-Airy-Strichartz}, and their successful applications in attacking the global wellposedness and scattering problems at mass- or energy- critical level \cite{Kenig-Merle:2006:focusing-energy-NLS-radial, Kenig-Merle:2007:focusing-energy-nonlinear-Wave, Kenig-Merle:2008:focusing-energy-nonlinear-Wave, Killip-Kwon-Shao-Visan:2008:minimal-mass-blow-up-solution-to-critical-gKdV, Killip-Tao-Visan:2008:cubic-NLS-radial, Killip-Visan:2008:focusing-energy-critical-NLS-higher-d, Killip-Visan-Zhang:2008:radial-NLS-hihger-d, Pausader-Shao:2009:GWP-L2-critical-4NLS-high-dimensions, Tao-Visan-Zhang:2007:radial-NLS-higher}. Roughly speaking, the profile decomposition investigates the general structure of a sequence of solutions to \eqref{eq:4thNLS} and aims to compensate for the loss of compactness of the solution operator caused by the natural symmetries of the equation. By passing to a subsequence, a sequence of solutions is expected to be written as a summation of the superposition of concentrating waves and a remainder (see Theorem \ref{thm:linear-profile}). The concentrating waves are referred to as ``profiles", which encode certain symmetry information of the equation and are orthogonal in some sense (see Remark \ref{re:profiles-orthogonal}); the remainder term is negligible in most applications.

The profile decomposition starts from a refinement of the Strichartz inequality. The usual Strichartz inequality \cite[p.38, Theorem 2.1]{Kenig-Ponce-Vega:1991:dispersive-estimates} asserts that
\begin{equation}\label{eq:symmetric-strichartz}
\|D_\mu^{1/3}S_\mu(t)f\|_{L^6_{t,x}(\mathbb{R}\times \mathbb{R})}\le C\|f\|_{L^2},
\end{equation} where $S_\mu(t)$ is the solution operator to Equation \eqref{eq:4thNLS} defined by
$$S_\mu(t)f(x):=e^{it(\Delta^2-\mu\Delta)}f(x):
=\int_\mathbb{R}e^{i(x\xi+t\phi_\mu(\xi)}\widehat{f}(\xi)d\xi,\,\phi_\mu(\xi)=\xi^4+\mu\xi^2;$$
and $D^{\alpha}_\mu$ with $\alpha\in \mathbb{R}$ is the nonhomogeneous differentiation operator for by
\begin{equation*}
D_\mu^{\alpha} f(x):=\int_\mathbb{R} e^{ix\xi}(\mu+6\xi^2)^\frac{\alpha}{2} \widehat{f}(\xi)d\xi.
\end{equation*}
We shall write $S(t)=S_0(t)$ and $D^\alpha=D^\alpha_0$. Note that Estimate \eqref{eq:symmetric-strichartz} also follows from \cite{BenArtzi-Koch-Saut:2000:dispersive-estimate-for-4th-NLS}, and of course from the refinement in Lemma \ref{le:refined-strichartz}. The primary reasons for us to study $D^{\alpha}_\mu$ is:  (1) to treat two interesting cases $\mu=0$ and $\mu>0$ in a same manner; (2) The oscillatory integral $$\int_\mathbb{R} e^{ix\xi+it\phi_\mu(\xi)}|\mu+6\xi^2|^{1/6} d\xi $$ on the left hand side of \eqref{eq:symmetric-strichartz} matches the form considered by Kenig, Ponce and Vega in \cite[p.38, (2.2)]{Kenig-Ponce-Vega:1991:dispersive-estimates} up to a constant multiple $2^{1/6}$, as $\phi_\mu^{\prime\prime}(\xi)=2(\mu+6\xi^2)$.

The estimate \eqref{eq:symmetric-strichartz} is not optimal within Besov spaces. We need the following refinement for our purpose.
\begin{lemma}\label{le:refined-strichartz} For any $p>1$ and $\mu\ge 0$
\begin{equation}\label{eq:refined-strichartz}
\|D_\mu^{1/3}S_\mu(t)f\|_{L^6_{t,x}(\mathbb{R}\times \mathbb{R})}
\le C\left(\sup_{\tau}|\tau|^{\frac 12-\frac 1p}\|\hat{f}\|_{L^p(\tau)}\right)^{1/3}\|f\|^{2/3}_{L^2},
\end{equation} where $\tau$ denotes an interval on the real line with the length $|\tau|$.
\end{lemma} We will adapt a proof from \cite{Shao:2008:linear-profile-Airy-Maximizer-Airy-Strichartz} and it will be proven in Section \ref{sec:strichartz-estimates}.

By using Lemma \ref{le:refined-strichartz} and certain improved localized restriction estimates in Lemma \ref{L:loaclrestrict} in Section \ref{sec:local-restriction-estimates}, we can prove the following theorem, which is the main result in this paper.

\begin{theorem}[Linear profile decomposition]\label{thm:linear-profile}
Let $\mu\ge 0$ and let $(u_n)_{n\ge 1}$ be a sequence of complex-valued functions satisfying $\|u_n\|_{L^2}\leq 1$.
Then up to a subsequence, for any $l\ge 1$, there exists a sequence of functions $(\phi^j)_{1\le j\le l}\in L^2$, $w_n^l\in L^2$ and a family of parameters, $(h_n^j,\xi_n^j, x_n^j, t_n^j)_{1\le j\le l,\atop n\ge 1}$, such that
\begin{equation}\label{eq:profile-decomposition}
u_n=\sum_{1\le j\le l, \xi^j_n\equiv 0, \atop \text{ or } |h_n^j\xi_n^j|\to \infty} S_{\mu}(t^{j}_n)g^j_n[e^{i(\cdot)h^j_n\xi^j_n}\phi^j]+w^l_n,
\end{equation}
where $g^j_n(\phi):=\frac{1}{(h^j_n)^{1/2}}\phi(\frac{x-x^j_n}{h^j_n})$. This decomposition enjoys the following properties:
\begin{equation}\label{eq:error-term}
\limsup_{l\to\infty}\limsup_{n\to\infty}\|D_{\mu}^{1/3}S_{\mu}(t)w^l_n\|_{L^6_{t,x}(\mathbb{R}
\times\mathbb{R})}=0,
\end{equation} and for $j\neq k$, $(h_n^j,\xi_n^j,x_n^j,t_n^j)_{n\ge 1}$ and $(h_n^k,\xi_n^k,x_n^k,t_n^k)_{n\ge 1}$ are pairwise orthogonal in the sense that,
\begin{align}\label{eq:ortho}
&\text{either }\limsup_{n\to\infty}(\frac{h^j_n}{h^k_n}+\frac{h^k_n}{h^j_n}+h^j_n|\xi^j_n-\xi^k_n|)=\infty,\\
&\text{or }(h^j_n,\xi^j_n)=(h^k_n,\xi^k_n)\;\text{ and }\\
&\nonumber \limsup_{n\to\infty}\frac{|t^k_n-t^j_n|}{(h^j_n)^4}+\frac{|(t^k_n-t^j_n)(\mu+6(\xi^j_n)^2)|}{(h^j_n)^2}+\frac{|x^j_n-x^k_n- 2(t^j_n-t^k_n)(2(\xi^j_n)^2+\mu)\xi_n^j|}{h^j_n}=\infty.
\end{align}
\end{theorem}

\begin{remark}[Orthogonality of profiles]\label{re:profiles-orthogonal}
The orthogonality condition on the parameters, $\{(h_n^j,\xi_n^j,x_n^j,t_n^j)\}$, is the origin of orthogonality for profiles. Under this condition, the profiles are separated either in the spatial space, or in the frequency space, or have very different scales, or are distant in time. In particular, we have, for any $l\ge 1$,
\begin{align}
\label{eq:L2-orthogonality} &\limsup_{n\to\infty}\Big(\|u_n\|^2_{L^2}-(\sum^l_{j=1}\|\phi^j\|^2_{L^2}+\|\omega^l_n\|^2_{L^2}) \Big)=0.\\
\label{eq:strichartz-space-orthogonality}
&\limsup_{n\to\infty}\bigl(\|\sum_{1\le j\le l} D_{\mu}^{1/3}S_{\mu}(t+t_n^j) g_n^j [e^{i(\cdot)h_n^j\xi_n^j}\phi^j]\|^6_{L^6_{t,x}} -\sum_{1\le j\le l}\| D_{\mu}^{1/3}S_{\mu}(t+t_n^j) g_n^j [e^{i(\cdot)h_n^j\xi_n^j}\phi^j]\|^6_{L^6_{t,x}}\bigr)=0.
\end{align}
\end{remark}

\begin{remark}[Lack of Galilean transform]\label{re:lack-Galilean}
In the decomposition \eqref{eq:profile-decomposition}, we have treated the high and low frequencies differently. This is essentially due to lack of Galilean transform for Equation \eqref{eq:4thNLS}. More precisely, a computation for $S$ reveals that,
$$S(t)[e^{i(\cdot)N}\phi](x)
=e^{ixN+itN^4}e^{it\Delta^2+4itN\partial_x^3+i6N^2t\partial_x^2}\phi(x+4tN^3).$$
The operator on the right hand side can not be expressed as a form of $S(t)$ and contains some mixed terms, $e^{4itN\partial_x^3+6itN^2\partial_x^2}$. In contrast, for the linear Schr\"odinger evolution operator $e^{it\Delta}$,
$$e^{it\Delta}[e^{i(\cdot)N}\phi](x)
=e^{ixN+itN^2}[e^{it\Delta}]\phi(x+2tN),$$ which heuristically says that, up to a modulation $e^{ixN+itN^2}$, the propagation of a high-frequency wave is a dislocation in spatial space of the propagation of a low frequency wave, which is the ``so-called" effect of Galilean transform.

However, in view of Proposition \ref{prop:convg-to-schr}, when $N\to \infty$, $S(t)[e^{i(\cdot)N}\phi]$ behaves like a second order Schr\"odinger solution, $e^{-it\Delta}\phi$. We will use it to compare the optimal constants of the Strichartz inequalities for both equations, see the argument of Theorem \ref{thm:dichotomy}. It is similar to a previous observation by Christ, Colliander and Tao in \cite{Christ-Colliander-Tao:2003:asymptotics-modulation-canonical-defocusing-eqs} that the solutions to Korteweg-de Vries equations (KdV) or modified KdV at high frequencies can be well approximated by those to nonlinear Schr\"odinger equation (NLS); this observation turns out to be very useful to explore various wellposedness/ill-posedness results between KdV and NLS equations; see also \cite{Tao:2007:two-remarks-on-gKdV}, \cite{Shao:2008:linear-profile-Airy-Maximizer-Airy-Strichartz} and \cite{Killip-Kwon-Shao-Visan:2008:minimal-mass-blow-up-solution-to-critical-gKdV}.
\end{remark}

The decomposition in Theorem \ref{thm:linear-profile} is similar to that in \cite{Shao:2008:linear-profile-Airy-Maximizer-Airy-Strichartz} for the Airy Equation, where lack of Galilean transform is the case and hence different frequencies are treated in different ways. The new difficulty here is a lack of scaling invariance when $\mu >0$; in other words, we can only take advantage of the spatial and temporal translations; this complicates the task of establishing orthogonality results for profiles which are essential for all purposes, see Lemma \ref{le:orthog-strichartz-space}.
\begin{remark}[A comparison with nonlinear wave equation (NLW) and NLS]\label{re:comparison-NLS-NLW}
Let us make a comparison with those for NLW and NLS.
\begin{itemize}
\item In \cite{Bahouri-Gerard:1999:profile-wave}, for energy critical nonlinear wave equations with $\dot{H}^1$-initial data in $\mathbb{R}^3$, Bahouri-G\'erard establish the following decomposition,
$$e^{it\sqrt{-\Delta}}u_n(x)=\sum_{j=1}^l \frac {1}{\sqrt{h_n^j}} e^{it\sqrt{-\Delta}} \phi^j(\frac {t-t_n^j}{h_n^j}, \frac {x-x_n^j}{h_n^j})+e_n^l. $$ There is no frequency parameter since modulation is not a symmetry in $\dot{H}^1$.

\item In \cite{Merle-Vega:1998:profile-schrod}, for mass critical nonlinear Schr\"odinger equation with $L^2$-initial data in $\mathbb{R}^2$, Merle-Vega obtain the following decomposition,
    $$e^{it\Delta}u_n(x)=\sum_{j=1}^l \frac {1}{h_n^j} e^{ix\xi_n^j}e^{it_n^j\Delta} \phi^j(\frac {t-t_n^j}{(h_n^j)^2}\frac {x-x_n^j}{h_n^j})+e_n^l. $$
    There is no difference between high-low frequencies thanks to the Galilean transform.
\end{itemize}
\end{remark}

The linear profile decomposition proves to be a very useful tool in understanding the global wellposedness and scattering problems to certain critical and supercritical nonlinear dispersive equations. It serves as the primary motivation to develop such decompositions in order to understand certain nonlinear analogue of Equation \eqref{eq:4thNLS}, for instance, see \cite{Pausader:2009:cubic-4th-NLS, Pausader:2009:focusing-4th-NLS-radial, Pausader-Shao:2009:GWP-L2-critical-4NLS-high-dimensions}. In \cite{Kenig-Merle:2006:focusing-energy-NLS-radial}, Kenig-Merle introduced the method of concentration-compactness/rigidity to study the global wellposedness and scattering problems for the focusing radial nonlinear Schr\"odinger equation at the energy critical regularity; a key ingredient is the linear profile decomposition developed by Keraani \cite{Keraani:2001:profile-schrod-H^1}, which is employed to obtain the existence of minimal-energy blow-up solution. Similar ideas of extracting minimal blow-up ``bubbles" appearing previously in the works of Bourgain and I-team (Colliander, Keel, Staffilani, Takaoka, Tao) \cite{Bourgain:1999:radial-NLS, I-teem:2008:GWP-for-energy-critical-NLS-in-R3} for energy-critical NLS in $\mathbb{R}^3$. For applications to the mass/energy-critical nonlinear Schr\"odinger equations, we refer readers to Killip-Visan's survey \cite{Killip-Visan:2008:clay-lecture-notes}.

\subsection{An application} In \cite{Shao:2008:linear-profile-Airy-Maximizer-Airy-Strichartz, Shao:2008:maximizers-Strichartz-Sobolev-Strichartz}, the third author used the linear profile decomposition to prove the existence of extremals for the Strichartz inequality for the Schr\"odinger equation in high dimensions. This approach can be viewed as a simplified manifestation of the concentration-compactness idea. In this paper we consider a similar ``extremisers" problem,
\begin{equation}\label{eq:extremisers}
\mathbf{S}:=\sup_{f\neq 0,\,\|f\|_{L^2}\le 1} \frac {\|D^{1/3}S(t)f\|_{L^6_{t,x}(\mathbb{R}\times \mathbb{R})}}{\|f\|_{L^2}}.
 \end{equation} Here $S(t):=S_0(t)$. We will establish a dichotomy result on existence of extremals for \eqref{eq:extremisers} by using the ``profile decomposition" tool.

In context of the Strichartz inequality for the Schr\"odinger equation in low dimensions, there are other methods to prove existence of extremals such as by an elaborate concentration-compactness method by Kunze \cite{Kunze:2003:maxi-strichartz-1d}, by two successive applications of the Cauchy-Schwarz inequality by Foschi \cite{Foschi:2007:maxi-strichartz-2d}, by developing a representation formula of the Strichartz inequality by Hundertmark, Zharnitsky \cite{Hundertmark-Zharnitsky:2006:maximizers-Strichartz-low-dimensions}, and by using the heat-flow deformation method by Bennett, Bez, Carbery, Hundertmark \cite{Bennett-Bez-Carbery-Hundertmark:2008:heat-flow-of-strichartz-norm}; also see \cite{Carneiro:2008:sharp-strichartz-norm}. Moreover Gaussians are proven to be extremals \cite{Foschi:2007:maxi-strichartz-2d, Hundertmark-Zharnitsky:2006:maximizers-Strichartz-low-dimensions, Bennett-Bez-Carbery-Hundertmark:2008:heat-flow-of-strichartz-norm}.

We first note that the solution map, $S(t)$, from $L^2$ to the Strichartz space is not compact: an arbitrary $L^2$ bounded sequence may not give rise to a strongly convergent subsequence in the Strichartz space. Indeed, that $S(t)$ fails to be compact can be easily seen by creating counterexamples of considering several explicit symmetries in $L^2$, e.g.,
\begin{itemize}
\item spatial translation, $u(t,x)\to u(t,x-x_0) $ for some $x_0\in \mathbb{R}$.
\item time translation, $u(t,x)\to u(t-t_0,x) $ for some $t_0\in \mathbb{R}$.
\item scaling, $u(t,x) \to \lambda^{-1/2}u(t/\lambda^4,x/\lambda)$ for some $\lambda>0$.
\item modulation, $f\to e^{ix\xi_0}f$ for some $\xi_0\in \mathbb{R}$.
\end{itemize}
However as an application of the profile decomposition in Theorem \ref{thm:linear-profile}, we are able to establish a dichotomy result on the existence of an extremal $f$ to the Strichartz inequality \eqref{eq:extremisers}.

\begin{theorem}\label{thm:dichotomy}
Either an extremiser exist for $\mathbf{S}$, or there exists a sequence of $a_n$ satisfying $\lim_{n\to\infty}|a_n|=\infty$ and $f\in L^2$ so that
$$\mathbf{S}=\lim_{n\to\infty}\frac {\|D^{1/3}S(t)[e^{ixa_n}f]\|_{L^6_{t,x}}}{\|f\|_{L^2}}.$$
Moreover, in the latter case, $\mathbf{S}=\mathbf{S}_{schr}$ where $\mathbf{S}_{schr}$ is the optimal constant for the Strichartz inequality for the Schr\"odinger equation defined by
\begin{equation}\label{eq:schr-max}
\mathbf{S}_{schr}:=\sup_{\phi\neq 0,\,\|\phi\|_{L^2}\le 1}\frac {\|e^{-it\Delta} \phi\|_{L^6_{t,x}}}{\|\phi\|_{L^2}};
\end{equation}
and $f$ can be identified as Gaussians up to the natural symmetries associated to \eqref{eq:schr-max}.
\end{theorem}

\begin{remark}We may test ${\|D^{1/3}S(t)f\|_{L^6_{t,x}}}/{\|f\|_{L^2}}$ against a few numerical examples such as $e^{-|x|^2}$ or $(1+|x|)^{-\alpha}$ for $\alpha>1/2$ to find out whether there would hold $\mathbf{S}>\mathbf{S}_{schr}$ in order to rule out the second alternative in Theorem \ref{thm:dichotomy}; we may also formulate an analogous statement for $S_\mu$ with $\mu>0$; but we will not pursue these interesting matters here.
\end{remark}

This paper is organized as follows. In Section \ref{sec:notations}, we introduce some notations. In Section \ref{sec:strichartz-estimates}, we prove Lemma \ref{le:refined-strichartz}. In section \ref{sec:local-restriction-estimates}, we prove certain localized restriction estimates. In Sections \ref{sec:the-decomposition} and \ref{sec:an-estimate}, by Lemmas \ref{le:refined-strichartz} and \ref{L:loaclrestrict}, we establish the linear profile decomposition theorem \ref{thm:linear-profile} for a sequence of functions $(u_n)_{n\ge 1}$ which are bounded in $L^2$. In Section \ref{sec:dichotomy-extremisers}, we establish the dichotomy result Theorem \ref{thm:dichotomy}.

\noindent{\bf ACKNOWLEDGEMENT.}
Part of this work was done during the first author's stay in Johns Hopkins University, he would like to thank people in department of Mathematics especially Christopher Sogge. Part of this work was supported while the second author stayed at IHP and he wants to thank this institution for its hospitality. S. Shao was supported by National Science Foundation under agreement No. DMS-0635607 during the early preparations of this work. Any opinions, findings and conclusions or recommendations expressed in this paper are those of the authors and do not reflect necessarily the views of the National Science Foundation.


\section{Notations}\label{sec:notations}
We use $X\lesssim Y$, $Y\gtrsim X$, or $X=O(Y)$ to denote the estimate $|X|\le
C Y$ for some constant $0<C<\infty$, which will not depend on the functions. If $X\lesssim Y$ and $Y\lesssim X$ we will write $X\sim Y$. If the constant $C$ depends on a special parameter, we shall denote it explicitly by subscripts.

We define the space-time norm $L^q_tL^r_x$ of $f$ on $\mathbb{R}\times \mathbb{R}$ by
$$\|f\|_{L^q_tL^r_x(\mathbb{R}\times\mathbb{R})}:=\left(\int_{\mathbb{R}}\left(\int_{\mathbb{R}}
|f(t,x)|^{r}d\,x\right)^{q/r}d\,t\right)^{1/q},$$ with the usual modifications when $q$ or $r$ are
equal to infinity, or when the domain $\mathbb{R}\times\mathbb{R}$ is replaced by a small space-time region. When $q=r$, we abbreviate it by $L^{q}_{t, x}$. Unless specified, all the space-time integrations are taken over $\mathbb{R}\times \mathbb{R}$, and all the spatial integrations over $\mathbb{R}$.

We fix the notation that $\lim_{n\to \infty}$ should be understood as $\limsup_{n\to \infty}$ throughout this paper.

The spatial Fourier transform is defined via
$$\widehat{f}(\xi):=\int_\mathbb{R} e^{-ix\xi}f(x)dx;$$
the space-time Fourier transform is defined analogously.

The inner product $\langle\cdot,\cdot\rangle_{L^2}$ in the Hilbert space $L^2$ is defined via $$\langle f,g\rangle_{L^2}:=\int_{\mathbb{R}}f(x)\overline{g}(x)dx,$$
where $\overline{g}$ denotes the usual complex conjugate of $g$ in the complex plane $\mathbb{C}$.

\section{The refinement of the Strihcartz inequality}\label{sec:strichartz-estimates}
In this section we prove Lemma \ref{le:refined-strichartz}. We first introduce the notion of Whitney decomposition
as in \cite{Killip-Visan:2008:clay-lecture-notes}.
\begin{definition}
Given $j\in \mathbb{Z}$, we denote by $\mathcal{D}_j$ the set of all dyadic intervals in $\mathbb{R}$ of
length $2^j$: $$\mathcal{D}_j:=\{2^j[k,k+1):k\in\mathbb{Z}\}.$$ We also write $\mathcal{D}:=\cup_{j\in \mathbb{Z}}\mathcal{D}_j$.
Given $I\in \mathcal{D}$, we define $f_I$ by $\widehat{f}_I=\widehat{f}\;1_I$ where $1_I$ denotes the
characteristic function on $I$.
\end{definition}
Given two distinct $\xi,\xi'\in\mathbb{R}$, there is a
unique maximal pair of dyadic intervals $I,\,I'\in\mathcal{D}$ such that
$$\xi\in I,\xi'\in I'\;,\;|I|=|I'|\;,\;{\rm dist}(I,I')\geq 4|I|,$$ where
${\rm dist}(I,I')$ denotes the distance between $I$ and $I'$, and $|I|$ denotes the length of the dyadic interval $I$. Let $\mathcal{F}$ denote all such pairs as $\xi\neq\xi'$ varies over $\mathbb{R}\times\mathbb{R}$. Then we have
\begin{equation}
\sum_{(I,I')\in\mathcal{F}}1_I(\xi)1_{I'}(\xi)=1, {\rm for\;\;a.e.}\; (\xi,\xi')\in \mathbb{R}\times\mathbb{R}.
\end{equation}
Since $I$ and $I'$ are maximal, ${\rm dist}(I,I')\leq 10 |I|$. This shows that for a given $I\in\mathcal{D}$,
there exists a bounded number of $I'$ so that $(I,I')\in \mathcal{F}$, i.e.
$$\forall I\in\mathcal{D}\;,\;\#\{I':(I,I')\in\mathcal{F}\}\lesssim 1.$$

\begin{proof}[Proof of Lemma \ref{le:refined-strichartz}.]
Given $p>1$, we normalize $\sup_{\tau\in\mathbb{R}}|\tau|^{1/2-1/p}\|\widehat{f}\|_{L^p(\tau)}=1$.
Then for all dyadic intervals $I\in\mathcal{D}$,
\begin{equation}\label{E:dyadicest1}
\int_{I}|\widehat{f}|^p d\xi\leq |I|^{1-p/2}.
\end{equation}
Let $\phi_\mu(\xi):=\xi^4+\mu\xi^2$. Then
$$
(D_\mu)^{1/3}_x S_\mu(t)f=\int_{\mathbb{R}}e^{it\phi_\mu(\xi)+ix\xi}|\mu+6\xi^2|^{1/6}\widehat{f}(\xi)d\xi,
$$
then
$$
|D^{1/3}_\mu S_\mu(t)f|^2=\int_{\mathbb{R}}\int_{\mathbb{R}}e^{it(\phi_\mu(\xi)-\phi_\mu(\eta))+ix(\xi-\eta)}
\widehat{f}(\xi)\overline{\widehat{f}}(\eta)|\mu+6\xi^2|^{1/6}|\mu+6\eta^2|^{1/6}d\xi d\eta.
$$
Squaring the left hand side of~\eqref{eq:refined-strichartz}, we see it suffices to prove
$$
\|\int_{\mathbb{R}}\int_{\mathbb{R}}e^{it(\phi_\mu(\xi)-\phi_\mu(\eta))+ix(\xi-\eta)}
\widehat{f}(\xi)\overline{\widehat{f}}(\eta)|\mu+6\xi^2|^{1/6}|\mu+6\eta^2|^{1/6}d\xi d\eta\|_{L^3_{t,x}}\lesssim\|\widehat{f}\|^{4/3}_{L^2}.
$$
Let $u=\phi_\mu(\xi)-\phi_\mu(\eta),\,v=\xi-\eta$. By using the Hausdorff-Young inequality in both $t$ and $x$, we then have
$$\begin{aligned}
&\|D^{1/3}_\mu S_\mu(t)f\|^2_{L^6_{t,x}}\\
&=\|\int_{\mathbb{R}}\int_{\mathbb{R}}e^{itu+ixv}
\widehat{f}(\xi)\overline{\widehat{f}}(\eta)\frac{|\mu+6\xi^2|^{1/6}|\mu+6\eta^2|^{1/6}}{|\phi_\mu^\prime(\xi)
-\phi_\mu^\prime(\eta)|}dudv\|_{L^3_{t,x}}\\
&\leq C(\int_{\mathbb{R}}\int_{\mathbb{R}}\bigl|\widehat{f}(\xi)|\mu+6\xi^2|^{1/6}\bigr|^{3/2}
\bigl|\widehat{f}(\eta)|\mu+6\eta^2|^{1/6}\bigr|^{3/2}\frac{d\xi d\eta}{|\phi_\mu^\prime(\xi)-\phi_\mu^\prime(\eta)|^{1/2}})^{2/3}
\end{aligned}$$
where $\phi_\mu^\prime(\xi)-\phi_\mu^\prime(\eta)=2(\xi-\eta)(\mu+2(\xi^2+\xi\eta+\eta^2))$.

We restrict to the case where $\xi, \eta\ge 0$ by symmetry; in this case, $$\frac{\left((\mu+6\xi^2)(\mu+6\eta^2)\right)^{1/4}}{\vert\phi_\mu^\prime(\xi)-
\phi_\mu^\prime(\eta)\vert^\frac{1}{2}}\lesssim
\frac{1}{\vert\xi-\eta\vert^\frac{1}{2}}.$$
Then it reduces to proving
$$
\int\int\frac{|\widehat{f}(\xi)\widehat{f}(\eta)|^{3/2}}{|\xi-\eta|^{1/2}}d\xi d\eta\lesssim \int |\widehat{f}|^2d\xi.
$$
In view of the above inequality, we thus assume $\widehat{f}\geq 0$ from now on. By Whitney decomposition we have
$$
\widehat{f}(\xi)\widehat{f}(\eta)=\sum_{I,I'\in\mathcal{F}}\widehat{f}_{I}(\xi)\widehat{f}_{I}(\eta)
,\;\;{\rm for\;\;a.e.}\;\; (\xi,\eta)\in\mathbb{R}\times\mathbb{R}$$ and
$$\forall\; (\xi,\eta)\in I\times I'\;\;{\rm with}\;\;(I,I')\in\mathcal{F},|\xi-\eta|\sim |I|.
$$
Choose a slightly larger dyadic interval containing both $I$ and $I'$ but still of length comparable to
that of $I$, and denote it again by $I$.  We have therefore reduced our problem to proving

\begin{equation}\label{E:dyadicest2}
\sum_{I\in\mathcal{D}}\frac{(\int\widehat{f}~^{3/2}_{I}d\xi)^2}{|I|^{1/2}}\lesssim \int\widehat{f}~^2d\xi.
\end{equation}

To prove~\eqref{E:dyadicest2}, we need a further decomposition to $f_I=\sum_{n\in\mathbb{Z}}f_{n,I}$,
here $f_{n,I}$ is
defined by
$$\widehat{f}_{n,I}=\widehat{f}\;1_{\{\xi:2^n|I|^{-1/2}\leq\widehat{f}(\xi)\leq 2^{n+1}|I|^{-1/2}\}}.$$
By the Cauchy-Schwartz inequality, for any $\varepsilon_1 >0$
$$
\Big(\int \widehat{f}_I~^{3/2}d\xi\Big)^2=\Big(\sum_{n\in\mathbb{Z}}\int\widehat{f}~^{3/2}_{n,I} d\xi\Big)^2
\lesssim_{\varepsilon_1}\sum_{n\in\mathbb{Z}}2^{|n|\varepsilon_1}\Big(\int\widehat{f}~^{3/2}_{n,I} d\xi\Big)^2.
$$
The $\varepsilon_1$ we need will be a number less than $\varepsilon$ in~\eqref{E:dyadiclp}.
By the convergence of geometric series,~\eqref{E:dyadicest2} is a consequence of the following
\begin{equation}\label{E:dyadiclp}
\sum_{I\in\mathcal{D}}\frac{(\int \widehat{f}~^{3/2}_{n,I} d\xi)~^2}{|I|^{1/2}}\lesssim
2^{-|n|\varepsilon}\int\widehat{f}~^2d\xi\;,\,\text{ for some}\,\varepsilon>0 \text{ and all }\,n>0.
\end{equation}

By the Cauchy-Schwartz inequality,
$$\Big(\int\widehat{f}~^{3/2}_{n,I}d\xi\Big)^2\lesssim \int\widehat{f}~^2_{n,I} d\xi \int f_{n,I}d\xi.$$
When $n\geq 0$, by the Chebyshev's inequality and~\eqref{E:dyadicest1},
$$\begin{aligned}
\int \widehat{f}_{n,I}d\xi & \lesssim 2^n|I|^{-1/2}|\{\xi:\widehat{f}(\xi)\geq 2^n|I|^{-1/2}\}|\\
&\lesssim 2^n |I|^{-1/2}\frac{\int \widehat{f}~^p d\xi}{2^{np}|I|^{-p/2}}\\
&\lesssim 2^{-|n|(p-1)}|I|^{1/2}
\end{aligned}$$
for any $p>1$. On the other hand, when $n<0$,
$$\int\widehat{f}_{n,I}d\xi\lesssim 2^n|I|^{-1/2}|I|=2^{-|n|}|I|^{1/2}.$$
Combining these estimates, there exist an $\varepsilon>0$ such that
$$
\sum_{I\in\mathcal{D}}\frac{(\int \widehat{f}~^{3/2}_{n,I}d\xi)^2}{|I|^{1/2}}\lesssim
2^{-|n|\varepsilon}\sum_{I\in\mathcal{D}}\int\widehat{f}~^2_{n,I}d\xi
$$
Interchanging the order of summation, we obtain
$$\sum_{I\in\mathcal{D}}\int \widehat{f}~^2_{n,I}d\xi =\sum_{j\in\mathcal{Z}}\sum_{I\in\mathcal{D}_j}\int \widehat{f}~^21_{\{\xi\in I:\widehat{f}\sim 2^{n-j/2}\}}d\xi=\int_{\mathbb{R}}\sum_{j:\widehat{f}\sim 2^{n-j/2}}\widehat{f}~^2d\xi\lesssim\int\widehat{f}~^2d\xi. $$
Thus we get~\eqref{E:dyadicest2} from above two inequalities.
\end{proof}

\section{Localized restriction estimates}\label{sec:local-restriction-estimates}
\begin{lemma}\label{L:loaclrestrict}
For $4<q<6$, $0\le \mu$ and $\widehat{G}\in L^{\infty}(B(\xi_0,R))$ for some $R>0$, we have
\begin{equation}\label{E:localrestrict1}
\|D_\mu^\frac{2}{q}S_\mu(t)G\|_{L^q_{t,x}}\leq C_{q,R}\|\widehat{G}\|_{L^{\infty}}.
\end{equation}
\end{lemma}

\begin{proof}
We may assume that $\xi_1,\,\xi_2\ge 0$ in the proof. Recalling that $\phi_\mu(\xi)=\xi^4+\mu \xi^2$, we observe that, for $1<r<\infty$ and $\mu\ge 0$,
\begin{equation}\label{Computation}
\begin{split}
\frac{\left[(\mu+6\xi_1^2)(\mu+6\xi_2^2)\right]^\frac{r^\prime}{2r}}{\vert\phi_\mu^\prime(\xi_1)-
\phi_\mu^\prime(\xi_2)\vert^{r^\prime-1}}&=\frac{1}{\left[2\vert \xi_1-\xi_2\vert\right]^{r^\prime-1}}\left(\frac{\left[(\mu+6\xi_1^2)(\mu+6\xi_2^2)\right]^\frac{1}{2}}
{2(\xi_1^2+\xi_1\xi_2+\xi_2^2)+\mu}\right)^{r^\prime-1}\\
&\lesssim \frac{1}{\vert \xi_1-\xi_2\vert^{r^\prime-1}}.
\end{split}
\end{equation}
Let $q=2r$ with $2<r<3.$ To prove ~\eqref{E:localrestrict1} is equivalent to proving
$$\begin{aligned}
&\Big\|\int_{B(\xi_0,R)}\int_{B(\xi_0,R)}e^{ix(\xi_1-\xi_2)+it(\phi_\mu(\xi_1)-\phi_\mu(\xi_2))}
|(\mu+6\xi_1^2)(\mu+6\xi_2^2)|^{1/q}\hat{G}(\xi_1)\bar{\hat{G}}(\xi_2)d\xi_1d\xi_2\Big\|_{L^r_{t,x}}\\
&\le C_{q,R}\|\hat{G}\|^2_{L^{\infty}(B(\xi_0,R))}
\end{aligned}$$
Let $u:=\xi_1-\xi_2\;,\;v:=\phi_\mu(\xi_1)-\phi_\mu(\xi_2)$ and denote the resulting image of $B(\xi_0,R)\times B(\xi_0,R)$ by $\Omega$ under change of variables. Since $r>2$, by the Hausdorff-Young inequality, we see the left hand side of the inequality above is bounded by
$$C\Big(\int_{\Omega}\Big||(\mu+6\xi_1^2)(\mu+6\xi_2^2)|^\frac{1}{2r}\frac{\hat{G}(\xi_1)\hat{G}(\xi_2)}
{|\phi_\mu^\prime(\xi_1)-\phi_\mu^\prime(\xi_2)|}\Big|^{r'}dudv\Big)^{1/r'}.$$ The constant $C>0$ is induced because of change of variables. Changing the variables back, we obtain
$$C\bigl(\int_{B(\xi_0,R)\times B(\xi_0,R)}\frac{|(\mu+6\xi_1^2)(\mu+6\xi_2^2)|^\frac{r^\prime}{2r}}{|\phi_\mu^\prime(\xi_1)-\phi_\mu^\prime(\xi_2)|^{r'-1}}
|\hat{G}(\xi_1)\hat{G}(\xi_2)|^{r'} d\xi_1 d\xi_2\bigr)^{1/r^\prime}.$$
We may restrict to the region where $0\le\xi_1\le\xi_2$. In this case, using \eqref{Computation}, we see that after a change of variables,
$$\int_{0}^{R}\int_0^{\xi_2}
\frac{1}{\vert\xi_1-\xi_2\vert^{r^\prime-1}}d\xi_1d\xi_2
\lesssim_R1.$$
Thus we obtain ~\eqref{E:localrestrict1}; the proof of this lemma is complete.
\end{proof}


\section{The Linear profile decomposition}\label{sec:the-decomposition}
By the refined Strichartz estimate \eqref{eq:refined-strichartz}, we extract the frequency and scaling parameters. It closely follows the approach in \cite{Carles-Keraani:2007:profile-schrod-1d, Shao:2008:linear-profile-Airy-Maximizer-Airy-Strichartz}.

\begin{lemma}\label{L:scalefre}
Let $(u_n)_{n\geq 1}$ be a sequence of complex valued functions with $\|u_n\|_{L^2}\leq 1$. Then up to a subsequence, for any $\delta>0$, there exists $N=N(\delta)$, a family of $(\rho^j_n,\xi^j_n)_{1\leq j\leq N}\in(0,\infty)\times \mathbb{R}$ and a family $(f^j_n)_{1\leq j\leq N\atop n\geq 1}$ of $L^2$ bounded sequences such that
\begin{equation}\label{E:presum}
u_n=\sum_{j=1}^{N}f^j_n+q^N_n
\end{equation}
and there exists a compact set $K=K(N)$ in $\mathbb{R}$, for every $1\leq j\leq N$,
\begin{equation}\label{E:preprofilesf}
 \sqrt\rho^j_n|\widehat{f}^j_n(\rho^j_n\xi+\xi^j_n)| \leq C_{\delta}1_{K}(\xi).
\end{equation}
Here the sequence $(\rho^j_n,\xi^j_n)$ satisfies that,  if $j\neq k$,
\begin{equation}\label{E:preconstraintone}
\lim_{n\to\infty}(\frac{\rho^j_n}{\rho^k_n}+\frac{\rho^k_n}{\rho^j_n}+\frac{|\xi^j_n-\xi^k_n|}{\rho^j_n})=\infty.
\end{equation}
 The remainder term satisfies, for any $N\ge 1$,
\begin{equation}\label{E:preremainder}
\lim_{n\to\infty}\|D_\mu^{1/3}S_\mu(t)q^N_{n}\|_{L^6_{t,x}}\leq\delta ,
\end{equation}
furthermore, for any $N\ge 1$, \begin{equation}\label{E:preorthogonal}
\lim_{n\rightarrow\infty}\Big(\|u_n\|^2_{L^2}-\Big(\sum_{j=1}^N\|f^j_n\|^2_{L^2}+\|q^N_n\|^2_{L^2} \Big) \Big)=0.
\end{equation}
\end{lemma}
\begin{proof}
For $\gamma_n=(\rho_n,\xi_n)\in(0,\infty)\times \mathbb{R}$, we define $G_n:L^2\to L^2$,
$$G_n(f)(\xi)=\sqrt{\rho_n}f(\rho_n\xi+\xi_n).$$
We will induct on the $L^6_{t,x}$ norm. If $\lim_{n\to\infty}\|D^{1/3}_\mu S_\mu(t)u_n\|_{L^6_{t,x}}\leq\delta$ (recall that $\lim_{n\to\infty}f_n$ is understood as $\limsup_{n\to\infty} f_n$ throughout this paper), then we are done. Otherwise, up to a subsequence, we may assume that, for all $n$ in this subsequence,
$$\|D_\mu^{1/3}S_\mu(t)u_n\|_{L^6_{t,x}}> \delta.$$
On the other hand, by Lemma \ref{le:refined-strichartz} with $p=4/3$, we see there exists a family of intervals $I^1_n:=[-\xi^1_n-\rho^1_n,-\xi^1_n+\rho^1_n]$ such that
$$\int_{I^1_n}|\widehat{u}_n|^{4/3}d\xi\geq C_1\delta^4(\rho^1_n)^{1/3} $$
where $C_1$ depends only on the constant in Lemma \ref{le:refined-strichartz}. While for any $A>0$,
$$
\int_{I^1_n\cap\{|\widehat{u}_n|>A\}}|\widehat{u}_n|^{4/3}
d\xi\leq A^{-\frac{2}{3}}\|\widehat{u}_n\|^2_{L^2}.
$$
Let $C_\delta:=(C_1/2)^{-3/2}\delta^{-6}$. Then
$$\int_{I^1_n\cap\{|\widehat{u}_n|\leq C_\delta(\rho_n^1)^{-1/2}\}}|\widehat{u}_n|^{4/3}d\xi\geq \frac{C_1}{2}\delta^4(\rho^1_n)^{1/3}.$$
From H\"older's inequality, we have
$$\int_{I^1_n\cap\{|\widehat{u}_n|\leq C_\delta(\rho_n^1)^{-1/2}\}}|\widehat{u}_n|^{4/3}d\xi\leq \Big(\int_{I^1_n\cap\{|\widehat{u}_n|\leq C_\delta(\rho_n^1)^{-1/2}\}}|\widehat{u}_n|^{2}d\xi\Big)^{2/3}(|I^1_n|)^{1/3}.$$
This yields $$\int_{I^1_n\cap\{|\widehat{u}_n|\leq C_\delta(\rho_n^1)^{-1/2}\}}|\widehat{u}_n|^2d\xi\geq C'\delta^6,$$ where $C'>0$ is some constant depending only on $C_1$.

Define $v^1_n$ and $\gamma^1_n$ by
$$\widehat{v}^1_n:=\widehat{u}_n 1_{I^1_n\cap\{|\widehat{u}_n|\leq C_{\delta}(\rho^1_n)^{-1/2}\}}
\;,\;\gamma^1_n:=(\rho^1_n,\xi^1_n).$$
Then $\|v^1_n\|_{L^2}\geq (C')^{1/2}\delta^3$. Also by definition of $G$, we have
$$|G^1_n(\widehat{v}^1_n(\xi))|=|(\rho^1_n)^{1/2}\widehat{v}^1_n(\rho^1_n\xi+\xi^1_n)|\leq C_{\delta}
1_{[-1,1]}(\xi).$$
We repeat the same argument with $u_n-v^1_n$ in place of $u_n$. At each step, the $L^2$ norm decreases by
at least $(C')^{1/2}\delta^3$. After $N:=N(\delta)$ steps, we obtain $(v^j_n)_{1\leq j\leq N}$ and
$(\gamma^j_n)_{1\leq j\leq N}$ so that
$$u_n=\sum_{j=1}^Nv^j_n+q^N_n,$$ $$\|u_n\|^2_{L^2}=\sum_{j=1}^N\|v^j_n\|^2_{L^2}+\|q_n^N\|^2_{L^2}.$$
The latter equality is due to the disjoint of support on the Fourier side. We also have the error term estimate
~\eqref{E:preremainder}
$$\|D_\mu^{1/3}S_\mu(t)q_n^N\|_{L^6_{t,x}}\leq\delta.$$
Next, we will re-organize the decomposition to get~\eqref{E:preconstraintone}. We say that $\gamma^j_n=(\rho^j_n,\xi^j_n)$
and $\gamma^k_n=(\rho_n^k,\xi_n^k)$ are orthogonal if
$$
\lim_{n\to\infty}(\frac{\rho^j_n}{\rho^k_n}+\frac{\rho^k_n}{\rho^j_n}+\frac{|\xi^j_n-\xi^k_n|}{\rho^j_n})=\infty.
$$
We define $f^1_n$ to be the summation of those $v^j_n$ whose $\gamma^j_n$'s are not orthogonal to $\gamma^1_n$.
Then take the least $j_0\in[2,N]$ such that $\gamma^{j_0}_n$ is orthogonal to $\gamma^1_n$; then we define
$f^2_n$ to be the summation of the those $v^i_n$ whose $\gamma^j_n$'s are orthogonal to $\gamma^1_n$ but not
to $\gamma^{j_0}_n$.  Repeating this argument a finite number times, we obtain~\eqref{E:presum}.
The decomposition gives~\eqref{E:preconstraintone} automatically. Also the supports on the Fourier side
are disjoint, and we have~\eqref{E:preorthogonal}. Now we want to check that, up to a subsequence,~\eqref{E:preprofilesf} holds.

By construction, those $v^j_n$'s collected in $f^1_n$ have $\gamma^j_n$'s not orthogonal to $\gamma^1_n$, i.e. for those $j$, we have
\begin{equation}\label{E:finiteconstrain}
\lim_{n\rightarrow\infty }\frac{\rho^j_n}{\rho^1_n}+\frac{\rho^1_n}{\rho^j_n}<\infty\;,\;
\lim_{n\rightarrow\infty}\frac{|\xi^j_n-\xi^1_n|}{\rho^j_n}<\infty
\end{equation}
To show~\eqref{E:preprofilesf}, it is sufficient to show that, up to a subsequence,
$G^1_n(\widehat{v}^j_n)$ is bounded by a
compactly supported and bounded function. This implies~\eqref{E:preprofilesf} with $j=1$ and other $j$'s will be handled similarly by passing to subsequences successively. By construction, $|G^j_n(\widehat{v}^j_n)|\leq C_\delta 1_{[-1,1]}.$
Also, we observe that
$$G^1_n(\widehat{v}^j_n)=G^1_n(G^j_n)^{-1}G^j_n(\widehat{v}^j_n)$$
$$G^1_n(G^j_n)^{-1}f(\xi)=\sqrt{\frac{\rho^1_n}{\rho^j_n}}f(\frac{\rho^1_n}{\rho^j_n}\xi
+\frac{\xi^1_n-\xi^j_n}{\rho^j_n})$$
which yields the desired estimates for $G^1_n(\widehat{v^j_n})$ by~\eqref{E:finiteconstrain}.
\end{proof}

Next we perform a further decomposition to each $f_n^j$ to extract the space and time parameters of the profiles. The procedure is to take weak limits of normalized $f_n^j$ in $n$ successively; the reminder term is easily seen to converge to zero in the weak sense, which will be made clear from below. Roughly speaking, since it concentrates nowhere after taking possible (maximum times) weak limits, we can show that it converges to zero in the Strichartz norm.
\begin{lemma}\label{L:spacetime}
Suppose an $L^2$-bounded sequence $(f_n)_{n\ge 1}$ satisfies
$$\sqrt{\rho_n}|\widehat{f}_n(\rho_n(\xi+(\rho_n)^{-1}\xi_n))|\leq \widehat{F}(\xi)$$ with
$\widehat{F}\in L^{\infty}(K)$ for some compact set $K$ in $\mathbb{R}$ independent of $n$.
Then up to a subsequence, there exists a family $(y^{\alpha}_n,s^{\alpha}_n)\in\mathbb{R}\times\mathbb{R}$ and a sequence $(\phi^{\alpha})_{\alpha\geq 1}$ of $L^2$ functions such that, if  $\alpha\neq\beta$, as $n\to \infty$,
\begin{equation}\label{E:spacetimeconstrain}
|s^{\alpha}_n-s^{\beta}_n|+|\frac{(6\xi_n^2+\mu)(s^{\alpha}_n-s^{\beta}_n)}{\rho_n^2}|+\left|y_n^{\alpha}-y_n^{\beta}-
\frac{(4\xi_n^2+2\mu)\xi_n(s^{\alpha}_n-s^{\beta}_n)}{\rho_n^3}\right| \to \infty.
\end{equation}
For every $M\geq 1$, there exists $e^{M}_n\in L^2$,
\begin{equation}\label{E:spacetimesum}
f_n(x)=\sum_{\alpha=1}^M\sqrt{\rho_n}\left(S_{\rho_n^{-2}\mu}(s_n^{\alpha})[e^{i(\cdot)\rho^{-1}_n\xi_n}
\phi^{\alpha}(\cdot)]\right)(\rho_nx-y^{\alpha}_n)+e^M_n(x)
\end{equation}
and
\begin{equation}\label{E:spacetimeremainder}
 \lim_{M\to\infty}\lim_{n\to \infty}\|D_{\mu}^{1/3}S_{\mu}(t)e^M_n\|_{L^6_{t,x}}=0.
\end{equation}
Furthermore, for any $M\geq 1$,
\begin{equation}\label{E:spacetimeorthogon}
\lim_{n\rightarrow\infty}\Big(\|f_n\|^2_{L^2}
-(\sum_{\alpha=1}^M\|\phi^{\alpha}\|^2_{L^2}+\|e^M_n\|^2_{L^2})\Big)=0
\end{equation}
\end{lemma}

\begin{proof} We will be sketchy on our proof, see \cite{Carles-Keraani:2007:profile-schrod-1d} or \cite{Shao:2008:linear-profile-Airy-Maximizer-Airy-Strichartz} for similar arguments in other contexts. Let $P:=(P_n)_{n\geq 1}$  with $$\widehat{P}_n(\xi)=\sqrt{\rho_n}\widehat{f}_n(\rho_n(\xi+(\rho_n)^{-1}\xi_n)).$$
Let $\mathcal{W}(P)$ be the set of weak limits of subsequences of $P$ in $L^2$ defined by
$$\mathcal{W}(P)=\{\omega-\lim_{n\rightarrow\infty}e^{-ix\rho^{-1}_n\xi_n}S_{\rho_n^{-2}\mu}(-s_n)
[e^{i(\cdot)\rho^{-1}_n\xi_n}P_n(\cdot)](x+y_n)\;{\rm in}\;L^2\;:(y_n,s_n)\in\mathbb{R}^2 \}$$
and
$$\mu(P):=\sup\{\|\phi\|_{L^2}:\phi\in\mathcal{W}(P)\}.$$
Then taking weak limits and imposing the orthogonality condition on the parameters  ~\eqref{E:spacetimeconstrain} repeatedly, we have the following decomposition
$$
P_n(x)=\sum_{\alpha=1}^Me^{-ix\rho^{-1}_n\xi_n}S_{\rho_n^{-2}\mu}(s_n^{\alpha})
[e^{i(\cdot)\rho^{-1}_n\xi_n}\phi^{\alpha}(\cdot)](x-y^{\alpha}_n)+P^M_n(x).
$$
We may assume that $\widehat{\phi^\alpha}, \widehat{P_n^M}$ are in $ L^\infty$ and of compact support.
Let $P^M:=(P^M_n)_{n\geq 1}$, then the weak convergence holds,
\begin{equation}\label{E:limitp}
\lim_{M\rightarrow\infty}\mu (P^M)=0.
\end{equation}
For any $M\geq 1$, we also have
$$\lim_{n\rightarrow\infty}\Bigl(\|f_n\|^2_{L^2}-\bigl(\sum_{\alpha=1}^M\|\phi^{\alpha}\|^2_{L^2}+\|P^M_n\|^2_{L^2}\bigr)
\Bigr)=0.$$
Recall that $f_n(x)=\sqrt{\rho_n}e^{ix\xi_n}P_n(\rho_n x)$, the decomposition ~\eqref{E:spacetimesum} follows after setting $e^M_n(x):=\sqrt{\rho_n}e^{ix\xi_n}P^M_n(\rho_nx).$ It remains to obtain the strong convergence of the error in the Strichartz norm
\begin{equation}\label{E:spacetimeremainder2}
\lim_{M\rightarrow\infty}\lim_{n\rightarrow\infty}\|D_\mu^{1/3}S_\mu(t)[\sqrt\rho_n
e^{iy\xi_n}P^M_n(\rho_ny)]\|_{L^6_{t,x}}=0.
\end{equation}Indeed, by scaling, the norm above is equal to
\begin{equation*}
\Vert D_{\rho_n^{-2}\mu}^\frac{1}{3}S_{\rho_n^{-2}\mu}(t)\left[e^{iy\rho_n^{-1}\xi_n}P_n^M\right]\Vert_{L^6_{t,x}}
\end{equation*}By interpolation,
\begin{equation*}
\begin{split}
\Vert D_{\rho_n^{-2}\mu}^\frac{1}{3}S_{\rho_n^{-2}\mu}(t)\left[e^{iy\rho_n^{-1}\xi_n}P_n^M\right]\Vert_{L^6_{t,x}}
\le C\Vert D_{\rho_n^{-2}\mu}^\frac{2}{q}S_{\rho_n^{-2}\mu}(t)&\left[e^{iy\rho_n^{-1}\xi_n}P_n^M\right]\Vert^{q/6}_{L^{q}_{t,x}}
\\
&\times\Vert S_{\rho_n^{-2}\mu}(t)\left[e^{iy\rho_n^{-1}\xi_n}P_n^M\right]\Vert^{1-q/6}_{L^\infty_{t,x}}
\end{split}
\end{equation*} for $4<q<6$. Let $\omega_n(t):=S_{\rho_n^{-2}\mu}(t)\left[e^{iy\rho_n^{-1}\xi_n}P_n^M\right]$. Then by Lemma \ref{L:loaclrestrict}, we see that
\begin{equation*}\label{Estimt-1}
\Vert D_{\rho_n^{-2}\mu}^\frac{2}{q}\omega_n\Vert_{L^q_{t,x}}\lesssim 1
\end{equation*}
for some $q<6$, which is uniform in $n$. Therefore to prove \eqref{E:spacetimeremainder2}, we reduce to prove that
\begin{equation}\label{Estimt2}
\lim_{M\to +\infty}\limsup_{n\to+\infty}\Vert \omega_n\Vert_{L^\infty_{t,x}}=0.
\end{equation}
Now we are going to deduce \eqref{Estimt2} from the claim
\begin{equation}\label{Estimt3}
\limsup_{n\to +\infty}\Vert \omega_n\Vert_{L^\infty_{t,x}}\lesssim_K\mu(P^M).
\end{equation}
Indeed, assume $\widehat{P}_n^M$ is supported by $K$ and set $\chi\in C^\infty_c(\mathbb{R})$ be even and such that $\chi=1$ on $K$, and $(t_n,y_n)$ be such that
$$\Vert \omega_n\Vert_{L^\infty_{t,x}}=\vert \omega_n(t_n,y_n)\vert.$$
Then $\omega_n$ is supported by $K+\rho_n^{-1}\xi_n$. So if $$\chi_n(x):=\chi(x-\rho_n^{-1}\xi_n),$$
    then it follows that $$\omega_n=\mathcal{F}^{-1}\bigl(\chi_n \mathcal{F}\omega_n\bigr),$$
    where $\mathcal{F}$ denotes the spatial Fourier transform. Then
    \begin{align*}
  \|\omega_n\|_{L^\infty}&=\left|\omega_n(t_n,y_n)\right|= \left| \mathcal{F}^{-1}\bigl(\chi_n \mathcal{F}\omega_n\bigr) (t_n,y_n)\right|\\
    &=\lim_{n\to\infty} \left|\frac {1}{\sqrt{2\pi}} \int \mathcal{F}^{-1} (\chi_n)(x)\omega_n(t_n,x-y_n)dx\right|\\
    &=\lim_{n\to\infty} \left|\frac {1}{\sqrt{2\pi}} \int \mathcal{F}^{-1} (\chi_n)e^{ix\rho_n^{-1}\xi_n}e^{-ix\rho_n^{-1}\xi_n}\omega_n(t_n,x-y_n)dx\right|\\
    &=\lim_{n\to\infty} \left|\frac {1}{\sqrt{2\pi}} \int \mathcal{F}^{-1} \bigl(\chi_n(\cdot-\rho_n^{-1}\xi_n)\bigr)e^{-ix\rho_n^{-1}\xi_n}\omega_n(t_n,x-y_n)dx\right|\\
    &=\lim_{n\to\infty}  \left|\frac {1}{\sqrt{2\pi}}\int \mathcal{F}^{-1} (\chi)e^{-ix\rho_n^{-1}\xi_n}\omega_n(t_n,x-y_n)dx\right|.
  \end{align*}
  We observe that the second integrand above is in form of  defining elements in $\mathcal{W}(P^M)$. Thus in the limit, by Cauchy-Schwarz, we see that it is bounded by
  $$ \|\mathcal{F}^{-1} (\chi)\|_{L^2} \mu(P^M),$$
which is the desired bound. Therefore it ends the proof.
\end{proof}

\begin{remark}\label{re:reduction}
In Lemma~\ref{L:spacetime}, we will make a useful reduction when $\lim_{n\to\infty}\rho^{-1}_n\xi_n=a$ is finite:
we will let $\xi_n\equiv 0$. This is possible since we can replace $e^{ix\rho^{-1}_n\xi_n}\phi^{\alpha}$ with $e^{ix\alpha}\phi^{\alpha}$ by putting the difference into error term, then we can regard $e^{ix\alpha}\phi^{\alpha}$ as a new $\phi^{\alpha}$.
\end{remark}

\begin{proof}[Proof of Theorem \ref{thm:linear-profile}.]Having Lemmas ~\ref{L:scalefre} and \ref{L:spacetime}, we are ready to prove Theorem~\ref{thm:linear-profile}. Let
$$
(h^j_n,\xi^j_n,x^{j,\alpha}_n,t^{j,\alpha}_n):=((\rho^j_n)^{-1},\xi^j_n,(\rho^j_n)^{-1}y^{j,\alpha}_n
,(\rho^{j}_n)^{-4}s^{j,\alpha}_n).
$$
Then we put all the error terms together,
\begin{equation}\label{E:finaldecomp}
u_n=\sum_{{1\leq j\leq N,\xi_n^j\equiv 0}\atop{\text{ or }
|h^j_n\xi^j_n|\rightarrow\infty}}\sum_{\alpha=1}^{M_j}
S_{\mu}(t^{j,\alpha}_n)g^{j,\alpha}_n[e^{i(\cdot)h^j_n\xi^j_n}\phi^{j,\alpha}]+\omega^{N,M_1,\cdots,M_N}_{n}
\end{equation}
where $g^{j,\alpha}_n(\phi)(x):=\frac{1}{(h^j_n)^{1/2}}\phi(\frac{x-x^{j,\alpha}_n}{h^j_n})$  and
$\omega^{N,M_1,\cdots,M_N}_{n}=\sum_{j=1}^Ne^{j,M_j}_n+q^{N}_n.$ We enumerate the pair $(j,\alpha)$ by $\omega$ satisfying
\begin{equation}\label{eq:re-label}
\omega(j,\alpha)<\omega(k,\beta)\;{\rm if}\; j+\alpha<k+\beta,\text{ or }j+\alpha=k+\beta\;{\rm and}\;j<k.
\end{equation}
After re-labeling, \eqref{E:finaldecomp} can be rewritten as
$$u_n=\sum_{{1\leq j\leq l,\xi_n^j\equiv 0}\atop{\text{ or }|h^j_n\xi^j_n|\rightarrow\infty}}S_{\mu}(t^{j}_n)g^{j}_n
[e^{i(\cdot)h^j_n\xi^j_n}\phi^{j}]+\omega^{l}_n,$$ where $\omega^{l}_n:=\omega^{N,M_1,\cdots,M_N}_{n}$ with
$l=\sum_{j=1}^N M_j.$

Now we begin to verify this decomposition satisfies those two properties in Theorem \ref{thm:linear-profile}. Firstly we can see that the family $(h^j_n,\xi^j_n,x^j_n,t^j_n)_{n\ge 1}$ is pairwise orthogonal in the sense of \eqref{eq:ortho} in Theorem \ref{thm:linear-profile}. Secondly, the remainder term $D^{1/3}_\mu S_\mu(t)\omega^{N,M_1\cdots,M_N}_n$ converges to zero in the
Strichartz norm $\|\cdot\|_{L^6_{t,x}}$. That is we have to prove that, in view of the enumeration defined in \eqref{eq:re-label},
\begin{equation}\label{E:lastremainder}
\lim_{n\rightarrow\infty}\|D^{1/3}_\mu S_\mu(t)\omega^{N,M_1\cdots,M_N}_n\|_{L^6_{t,x}}\rightarrow 0
,\;{\rm as}\;\inf_{1\leq j\leq N}\{N,j+M_j\}\rightarrow\infty.
\end{equation} This is a crucial step, which is done by using the following Lemma \ref{le:orthog-strichartz-space} on orthogonality of profiles in the Strichartz space. One can also consult similar proofs in \cite{Keraani:2001:profile-schrod-H^1, Shao:2008:linear-profile-Airy-Maximizer-Airy-Strichartz}.
\end{proof}

\begin{lemma}\label{le:orthog-strichartz-space}
Let $(h_n^j,\xi_n^j,x_n^j,t_n^j)_{n\ge 1}$ be a family of orthogonal sequences. Let
$$Q_n^j(t,x):=D^{1/3}_{\mu} S_{\mu}(t+t^j_n)g^j_n[e^{i(\cdot)h^j_n\xi^j_n}\phi^j(\cdot)](x)$$
Then for every $l\ge 1$,
\begin{equation}
\lim_{n\rightarrow\infty}\Big( \|\sum_{j=1}^lQ_n^j\|^6_{L^6_{t,x}}-\sum_{j=1}^l\|Q_n^j\|^6_{L^6_{t,x}}\Big)=0
\end{equation} with $\xi^j_n\equiv 0$ when $\lim_{n\rightarrow\infty}|h^j_n\xi^j_n|<\infty$.
\end{lemma}
We present the proof of this lemma in the following section.

\section{Proof of Lemma \ref{le:orthog-strichartz-space}}\label{sec:an-estimate}
By an application of H\"older's inequality, the claim in Lemma \ref{le:orthog-strichartz-space} reduces to the following lemma,
\begin{lemma}For $j\neq k$,
    \begin{equation}\label{eq-61}
    \lim_{n\to\infty}\| Q_n^jQ_n^k\|_{L^3_{t,x}}=0,
    \end{equation}
    where $$ Q_n^j(t,x):= D^{1/3}_{\mu}S_{\mu}(t+t_n^j)g_n^j[e^{i(\cdot)h_n^j\xi_n^j}](x).$$
    Likewise for $Q_n^k$ and the parameters satisfy
\begin{align}
&\label{eq-62} \text{either }\frac {h_n^j}{h_n^k}+\frac {h_n^k}{h_n^j} +h_n^j|\xi_n^j-\xi_n^k|\to\infty, \\
&\label{eq-63} \text{or }(h_n^j,\xi_n^j)=(h_n^k,\xi_n^k) \text{ and }\\
&\nonumber \frac {|t_n^j-t_n^k|}{(h_n^j)^4}+\frac {|(t_n^j-t_n^k)(\mu+6(\xi_n^j)^2)|}{(h_n^j)^2}+\frac {|x_n^j-x_n^k-(t_n^j-t_n^k)(4(\xi_n^j)^3+2\mu\xi_n^j)|}{h_n^j}\to\infty.
\end{align}
\end{lemma}

\begin{proof}With no loss of generality, we may assume $\widehat{\phi^j}, \widehat{\phi^k} \in L^\infty (-1,1)$. We will prove \eqref{eq-61} case by case.

\textbf{Case 1.} Assume \eqref{eq-62}; we may first assume that $\frac {h_n^j}{h_n^k}\to 0.$ We rewrite $Q_n^j$ out,
\begin{equation}\label{eq-64}
Q_n^j=(h_n^j)^{1/2}\int e^{i(x+x_n^j)\cdot\xi+i(t+t_n^j)\bigl(\xi^4+\mu\xi^2\bigr)}\bigl(6\xi^2+\mu\bigr)^{1/6}
\widehat{\phi^j}\bigl(h_n^j(\xi-\xi_n^j)\bigr)d\xi.
\end{equation}
Likewise for $Q_n^k$. Following the Hausdorff-Young inequality, it reduces to show that the following:
\begin{align}\label{eq-65}
(h_n^jh_n^k)^{3/4}\iint \frac {|6\xi^2 +\mu|^{1/4}|6\eta^2 +\mu|^{1/4}}{|\eta-\xi|^{1/2} |2(\xi^2+\eta^2+\xi\eta)+\mu|^{1/2}} \left|\widehat{\phi^j}\bigl(h_n^j(\xi-\xi_n^j)\bigr)\widehat{\phi^k}\bigl(h_n^k(\xi-\xi_n^k)\bigr)
\right|^{3/2} d\xi d\eta \to 0,
\end{align}as $n\to\infty$. We may also assume that $\xi,\eta\ge 0$ in \eqref{eq-65}. Because $$\widehat{\phi^j}, \widehat{\phi^k}\in L^\infty(-1,1), \quad \frac {|6\xi^2 +\mu|^{1/4}|6\eta^2 +\mu|^{1/4}}{ |2(\xi^2+\eta^2+\xi\eta)+\mu|^{1/2}} \lesssim 1, $$ it is further reduced to showing that
\begin{equation}\label{eq-66}
(h_n^jh_n^k)^{3/4}\int_{\eta=\xi_n^k+O(\frac {1}{h_n^k})} \int_{\xi=\xi_n^j+O(\frac {1}{h_n^j})}\frac {1}{|\eta-\xi|^{1/2}}d\xi d\eta\to 0.
\end{equation} Since $\sqrt{a+h}-\sqrt{a-h}=\frac {2h}{\sqrt{a+h}+\sqrt{a-h}}\le 2\sqrt{h}$, we see that \eqref{eq-66} is bounded above by
$$ C(h_n^jh_n^k)^{3/4}  \times\frac {1}{h_n^k}\times (h_n^j)^{-1/2}\le C\bigl(\frac {h_n^j}{h_n^k}\bigr)^{1/4}\to 0.$$

Next we will assume that \begin{equation}\label{eq-67}
h_n^j=h_n^k, h_n^j|\xi_n^j-\xi_n^k|\to\infty.
 \end{equation}By the same reasoning as above, we aim to show that \eqref{eq-65} holds. Because of \eqref{eq-67}, either $h_n^j\xi_n^j, h_n^j\xi_n^k\to\infty$ or just one goes to infinity. In either case, the support information gives that
 $$ |\xi-\eta| \sim |\xi_n^j-\xi_n^k|.$$
Hence we see that \eqref{eq-65} is bounded by
$$ (h_n^j)^{3/2} (h_n^j)^{-2} |\xi_n^j-\xi_n^k|^{-1/2}=\left|h_n^j(\xi_n^j-\xi_n^k)\right|^{-1/2} \to 0. $$

\textbf{Case 2.} Assume the condition \eqref{eq-63}. We set
\begin{align*}
\mu_n &:=\mu(h_n^j)^2,\, a_n:=h_n^j\xi_n^j,\,b_n:=a_n^2+\mu_n,\\
y_n&:=\frac {x_n^j-x_n^k-(t_n^j-t_n^k)(4(\xi_n^j)^3+2\mu\xi_n^j)}{h_n^j}\\
&=\frac {x_n^j-x_n^k}{h_n^j}-\frac {t_n^j-t_n^k}{(h_n^j)^4}\times 2a_n(2a_n^2+\mu_n),\\
s_n&:=\frac{t_n^j-t_n^k}{(h_n^j)^4},\\
Y(s)&:= 2sa_n(2a_n^2+\mu_n).
\end{align*}

\textbf{Case 2a.} Assume that $a_n\to\infty$. By changing variables
$$y-\frac {x_n^k}{h_n^j}\to y,\,s+\frac {t_n^k}{(h_n^j)^4}\to s$$
followed by another change of variables, $s\to -s$, we see that it suffices to prove
\begin{equation}\label{eq-68}
\|I^j_n I^k_n\|_{L^3_{s,y}}\to 0, \text{ as } n\to\infty
\end{equation}
provided that
\begin{equation*}
|s_n|b_n\to \infty, \text{ or } |y_n|\to \infty.
\end{equation*}
 Here \begin{align*}
&I^j_n:= \int e^{i(y-\frac {x_n^j-x_n^k}{h_n^j})\xi+i(-s+\frac{t_n^j-t_n^k}{(h_n^j)^4})\bigl(\xi^4+
\mu_n\xi^2\bigr)}\bigl(6\xi^2+\mu_n\bigr)^{1/6} \widehat{\phi^j}(\xi-a_n)d\xi,\\
& I_n^k:=\int e^{iy\eta+i(-s)\bigl(\eta^4+
\mu_n\eta^2\bigr)}\bigl(6\eta^2+\mu_n\bigr)^{1/6} \widehat{\phi^k}(\eta-a_n)d\eta.
\end{align*}
From the stationary phase estimates \cite{Sogge:1993:fourier-analysis-book, Stein:1993}, there always holds that
\begin{equation}\label{bbd-1}
\begin{split}
|I_n^j|&\lesssim \min\{b_n^{1/6}, |s-s_n|^{-1/2}b_n^{-1/3}\},\\
|I_n^k|&\lesssim \min\{b_n^{1/6}, |s|^{-1/2}b_n^{-1/3}\}.
\end{split}
\end{equation}
This induces the following decomposition in the spatial space
\begin{equation}\label{eq-69}
\begin{split}
A_s&:= \{y: |y-Y(s)|\le |s|b_n\},\\
B_s&:= \{y: |y-Y(s)-y_n|\le |s-s_n|b_n\},\\
C_s&:=\mathbb{R} \setminus (A_s\cup B_s).
\end{split}
\end{equation}
We also split the time space into $$\mathbb{R}=\tau_0\cup \tau_n\cup (\tau_0\cup \tau_n)^c, \tau_0:=(-b_n^{-1}, b_n^{-1}), \tau_n:=(s_n-b_n^{-1},s_n+b_n^{-1}).$$

\textbf{Case 2aI.} We assume that $|s_n|b_n\to \infty$; an easy observation is that, for any $C>0$, $|s_n|\ge Cb_n^{-1}$ as long as $n$ is taken sufficiently large. We may also assume that $s_n\ge 0$. We first deal with the integral on $\mathbb{R}\times A_s$, for which we use the bound
\begin{equation}\label{eq-70}
|I_n^jI_n^k|\le C|s|^{-1/2}|s-s_n|^{-1/2}b_n^{-2/3}.
\end{equation}
Then since $|s_n|\gg b_n^{-1}$,
\begin{align*}
\int_{\tau_0}\int_{x\in A_s} |I_n^jI_n^k|^3 dsdx&\lesssim b_n^{-2}\int_{\tau_0}\int_{x\in A_s} |s|^{-3/2}|s-s_n|^{-3/2} dsdx\\
&\lesssim b_n^{-1}s_n^{-3/2} \int_{\tau_0} |s|^{-1/2}ds\\
&\lesssim (b_ns_n)^{-3/2} \to 0;
\end{align*}
and
\begin{align*}
\int_{\tau_n}\int_{x\in A_s} |I_n^jI_n^k|^3 dsdx&\lesssim b_n^{-2}\int_{\tau_n}\int_{x\in A_s} |s|^{-3/2}|s-s_n|^{-3/2} dsdx\\
&\lesssim b_n^{-1}\int_{\tau_n} |s|^{-1/2}|s-s_n|^{-3/2}ds\\
&\lesssim b_n^{-1}s_n^{-1/2} \int_{\tau_n}|s-s_n|^{-3/2}ds\\
&\le C(b_ns_n)^{-1/2} \to 0;
\end{align*}
and
\begin{align*}
\int_{(\tau_0\cup \tau_n)^c}\int_{x\in A_s} |I_n^jI_n^k|^3 dsdx&\lesssim b_n^{-2}\int_{(\tau_0\cup \tau_n)^c}\int_{x\in A_s} |s|^{-3/2}|s-s_n|^{-3/2} dsdx\\
&\lesssim b_n^{-1}\left(\int_{-\infty}^{-b_n^{-1}}+\int_{b_n^{-1}}^{s_n-b_n^{-1}}+\int_{s_n+b_n^{-1}}^\infty\right) |s|^{-1/2}|s-s_n|^{-3/2}ds\\
&=:I_1+I_2+I_3.
\end{align*}
Since
\begin{align*}
I_1&\lesssim b_n^{-1}b_n^{1/2}\int_{-\infty}^{-b_n^{-1}}|s-s_n|^{-3/2} ds\lesssim (b_ns_n)^{-1/2},\\
I_2&\lesssim b_n^{-1}b_n^{1/2}\int_{b_n^{-1}}^{s_n-b_n^{-1}}|s-s_n|^{-3/2} ds\lesssim (b_ns_n)^{-1/2},\\
I_3&\lesssim b_n^{-1}s_n^{-1/2}\int_{s_n+b_n^{-1}}^\infty |s-s_n|^{-3/2}ds\lesssim (b_ns_n)^{-1/2}.
\end{align*}
Hence
\begin{equation*}
\int_{(\tau_0\cup \tau_n)^c}\int_{x\in A_s} |I_n^jI_n^k|^3 dsdx \lesssim  (b_ns_n)^{-1/2} \to 0.
\end{equation*}
Since the bound \eqref{bbd-1} is symmetric with respect to $\tau_0$ and $\tau_n$, the estimate on $\mathbb{R}\times B_s$ follows similarly. So we reduce it to that on $\mathbb{R}\times (A_s \cup B_s)^c$, for which we use the following non-stationary bound for $I_n^j$ and $I_n^k$:
\begin{equation}\label{bbd-2}
|I_n^j|\le C\frac {b_n^{1/6}}{|y-Y(s)-y_n|},\,|I_n^k|\le C \frac {b_n^{1/6}}{|y-Y(s)|};
\end{equation}
We estimate $\int_{\tau_n}\int_{C_s} |I_n^jI_n^k|^3dyds$: by \eqref{bbd-1} and \eqref{bbd-2}, we have $$|I_n^jI_n^k|=\frac {b_n^{1/3}} {|y-Y(s)|}.$$
Then
\begin{equation*}
\begin{split}
\int_{\tau_n}\int_{C_s} |I_n^jI_n^k|^3dyds \le Cb_n \int_{\tau_n} \int_{|y-Y(s)|\ge b_ns} |y-Y(s)|^{-3} dyds&
\le Cb_n^{-1} \int_{(s_n-b_n^{-1}, s_n+b_n^{-1})} s^{-2} ds \\
&\le C (b_ns_n)^{-2}\to 0.
\end{split}
\end{equation*}Similarly we can estimate $\int_{\tau_0}\int_{C_s} |I_n^jI_n^k|^3dyds$. To estimate $\int_{(\tau_0\cup \tau_n)^c}\int_{C_s} |I_n^jI_n^k|^3dyds$, we use the following bound
$$ |I_n^jI_n^k|\le \frac {b_n^{-1/6}}{|y-Y(s)-y_n||s|^{1/2}}.$$
Then
\begin{equation*}
\int_{C_s} |I_n^jI_n^k|^3 dyds \lesssim Cb_n^{-1/2}|s|^{-3/2} \int_{|y-Y(s)-y_n|\ge b_n|s-s_n|} \frac {ds}{|y-Y(s)-y_n|^3}\le Cb_n^{-5/2} |s|^{-3/2} |s-s_n|^{-2}.
\end{equation*}Then for $n$ large enough such that $s_n\gg b_n^{-1}$, we split $$(\tau_0\cup \tau_n)^c =(-\infty,-b_n^{-1})\cup (b_n^{-1}, s_n/2)\cup (s_n/2, s_n-b_n^{-1})\cup (s_n+b_n^{-1},\infty).$$
Then on each interval we will show that the convergence holds:
\begin{equation*}
\begin{split}
C\int_{(-\infty,-b_n^{-1})} b_n^{-5/2} |s|^{-3/2} |s-s_n|^{-2} ds &\le Cb_n^{-5/2} s_n^{-2}\int_{(-\infty,-b_n^{-1})} |s|^{-3/2} ds \le C (b_ns_n)^{-2}.\\
C\int_{(b_n^{-1},s_n/2)} b_n^{-5/2} |s|^{-3/2} |s-s_n|^{-2} ds &\le Cb_n^{-5/2}s_n^{-2} \int_{(b_n^{-1},\infty)} s^{-3/2} ds \le C(b_ns_n)^{-2}.\\
C\int_{(s_n/2,s_n-b_n^{-1})} b_n^{-5/2} |s|^{-3/2} |s-s_n|^{-2} ds &\le Cb_n^{-5/2}s_n^{-3/2}\int_{(-s_n/2,-b_n^{-1})}|s|^{-2} ds \le C(b_ns_n)^{-3/2}.\\
C\int_{(s_n+b_n^{-1},\infty)} b_n^{-5/2} |s|^{-3/2} |s-s_n|^{-2} ds &\le Cb_n^{-5/2} s_n^{-3/2} \int_{(b_n^{-1},\infty)} |s|^{-2} ds \le C(b_ns_n)^{-3/2}.
\end{split}\end{equation*}
This finishes the proof on the region $\mathbb{R}\times (A_s\cup B_s)^c$ and therefore the proof for \textbf{Case 2aI}.

\textbf{Case 2aII.} We assume that $|s_n|b_n\le C_0$ for some fixed $C_0>0$ and $|y_n|\to\infty$. We first deal with the integration over $\mathbb{R}\times A_s$: fixing a large $K \gg C_0$, we split $\mathbb{R}:=\{s:\, b_n |s|\ge K\}\cup \{s:b_n |s|<K\}$. Then invoking the bound \eqref{bbd-1} that
$$ |I_n^jI_n^k| \lesssim b_n^{-2/3} |s|^{-1/2} |s-s_n|^{-1/2},$$
and $|s_n|\le \frac {C_0}{b_n}\ll \frac {K}{b_n}\le |s|$, which yields that $|s-s_n|\sim |s|$, and recalling that $|A_s|\le b_n|s|$, we have
\begin{align*}
\int_{\{s: \,b_n |s|\ge K\} }\int_{A_s} |I_n^jI_n^k|^3 dyds & \lesssim b_n^{-2} \int_{\{s:\, b_n |s|\ge K\} }\int_{A_s} |s|^{-3/2} |s-s_n|^{-3/2} dyds\\
&\lesssim b_n^{-1} \int_{\{s:\, b_n |s|\ge K\} } |s|^{-2} ds \lesssim K^{-1},
\end{align*} which is uniform in all large $n$ and  is going to zero as $K$ goes to infinity. On the other hand, on $\{s:\,b_n |s|<K\}\times A_s$, $|y-Y(s)|\le |s|b_n\le K\ll |y_n|$ for $n$ large enough, we then invoke the bound \eqref{bbd-2} for $I_n^j$ and \eqref{bbd-1} on $I_n^k$,
$$ |I_n^jI_n^k|\lesssim \frac {b_n^{-1/3}}{|s|^{1/2}} \frac {b_n^{1/6}}{|y-Y(s)-y_n|}\lesssim b_n^{-1/6} |y_n|^{-1}|s|^{-1/2}.$$
Then
\begin{align*}
\int_{\{s: \,b_n |s|< K\} }\int_{A_s} |I_n^jI_n^k|^3 dyds & \lesssim b_n^{-1/2}|y_n|^{-3} \int_{\{s:\, b_n |s|<K\} }\int_{A_s} |s|^{-3/2} dyds\\
&\lesssim b_n^{1/2} |y_n|^{-3} \int_{\{s:\, b_n |s|< K\} } |s|^{-1/2} ds \lesssim K^{1/2} |y_n|^{-3} ,
\end{align*} which is uniform in all large $n$ and  is going to zero as $K$ goes to infinity too. Similarly one can obtain similar results on $\mathbb{R}\times B_s$.

Now we come to the integration over $\mathbb{R}\times C_s$. We use the bound \eqref{bbd-2} for $I_n^j$ and \eqref{bbd-1} for $I_n^k$,
$$|I_n^jI_n^k|\lesssim b_n^{1/3} |y-Y(s)|^{-1},$$ then
\begin{align*}
\int_{\{s: \,b_n |s|\ge K\} }\int_{C_s} |I_n^jI_n^k|^3 dyds & \lesssim b_n \int_{\{s:\, b_n |s|\ge K\} }\int_{(A_s)^c} |y-Y(s)|^{-3} dyds\\
&\lesssim b_n \int_{\{s:\, b_n |s|\ge K\} } |b_ns|^{-2} ds \lesssim K^{-1},
\end{align*} which is uniform in all large $n$ and  is going to zero as $K$ goes to infinity. On the region $\{s:\, b_n|s|<K\}$, we use the bound \eqref{bbd-2} for $|I_n^jI_n^k|^{3/2}$ and \eqref{bbd-1} for $|I_n^jI_n^k|^{3/2}$,
$$|I_n^jI_n^k|\lesssim b_n^{1/3} |y-Y(s)|^{-1/2}|y-Y(s)-y_n|^{-1/2}.$$
To integrate over $C_s$ in space variable, fixing $s$ satisfying $|s|b_n <K$, we split $C_s:=(-\infty, Y(s)-|s|b_n) \cup (Y(s)+|s|b_n, Y(s)+y_n-|s|b_n)\cup (Y(s)+y_n+|s|b_n, \infty)$; those intervals are disjoint for large enough $n$ since $|s|b_n<K\ll y_n$ (note that we may assume that $y_n>0$). Then
\begin{align*}
&\int_{\{s: \,b_n |s|<K\} }\int_{-\infty}^{Y(s)-b_n|s|} |I_n^jI_n^k|^3 dyds \\
& \lesssim b_n \int_{\{s: \,b_n |s|<K\} } \int_{-\infty}^{Y(s)-b_n|s|} |y-Y(s)|^{-3/2}|y-Y(s)-y_n|^{-3/2} dyds \\
&\lesssim b_n \int_{\{s: \,b_n |s|<K\} } \int_{-\infty}^{Y(s)-b_n|s|}|y-Y(s)|^{-3/2} y_n^{-3/2}dyds \\
&\lesssim y_n^{-3/2} K^{1/2}\to 0, \text{ as } n\to \infty;
\end{align*}
and
\begin{align*}
&\int_{\{s: \,b_n |s|<K\} }\int_{Y(s)+y_n+|s|b_n}^\infty |I_n^jI_n^k|^3 dyds \\
& \lesssim b_n \int_{\{s: \,b_n |s|<K\} }\int_{Y(s)+y_n+|s|b_n}^\infty |y-Y(s)|^{-3/2}|y-Y(s)-y_n|^{-3/2} dyds \\
&\lesssim b_n \int_{\{s: \,b_n |s|<K\} }\int_{Y(s)+y_n+|s|b_n}^\infty |y-Y(s)-y_n|^{-3/2} y_n^{-3/2}dyds \\
&\lesssim y_n^{-3/2} K^{1/2}\to 0, \text{ as } n\to \infty.
\end{align*}
While for the integration over the middle interval, we split it into even smaller intervals,
$$(Y(s)+|s|b_n, Y(s)+y_n-|s|b_n)=(Y(s)+|s|b_n, Y(s)+y_n/2)\cup (Y(s)+y_n/2, Y(s)+y_n-|s|b_n) $$
then
\begin{align*}
&\int_{\{s: \,b_n |s|<K\} }\int_{Y(s)+b_n|s|}^{Y(s)+y_n-|s|b_n} |I_n^jI_n^k|^3 dyds\\
& \lesssim b_n \int_{\{s: \,b_n |s|<K\} }\bigl(\int_{Y(s)+b_n|s|}^{Y(s)+y_n/2}+\int_{Y(s)+y_n/2}^{Y(s)+y_n-b_n|s|} \bigr) |y-Y(s)|^{-3/2}|y-Y(s)-y_n|^{-3/2} dyds \\
&\lesssim b_n y_n^{-3/2}\int_{\{s: \,b_n |s|<K\} } \bigl(|b_ns|^{-1/2}-(y_n-b_n|s|)^{-1/2}\bigr)ds \\
&\lesssim b_ny_n^{-3/2} \int_{\{s: \,b_n |s|<K\} } |b_ns|^{-1/2}ds\\
&\lesssim K^{1/2}y_n^{-3/2}\to 0, \text{ as } n\to \infty.
\end{align*}
This finishes the proof for \textbf{Case 2aII}, thus \textbf{Case 2a}.

\textbf{Case 2b.} We are left with the case when $\xi_n^j=\xi_n^k\equiv 0$.  In this case, the orthogonality condition becomes
\begin{equation}\label{eq-71}
\frac {|t_n^j-t_n^k|}{(h_n^j)^4}\to\infty, \text{ or }\frac {|(t_n^j-t_n^k)\mu_n|}{(h_n^j)^4}\to \infty, \text{ or } \frac {|x_n^j-x_n^k|}{h_n^j}\to\infty.
\end{equation} This case can be similarly handled as in \textbf{Case 2a}; we omit the details.
\end{proof}

\section{A dichotomy on extremisers}\label{sec:dichotomy-extremisers}
We simplify the approach in \cite{Shao:2008:linear-profile-Airy-Maximizer-Airy-Strichartz} and present the following argument when $\mu=0$, also see \cite{Killip-Visan:2008:clay-lecture-notes}.
\begin{proof}
Choose an extremising sequence of functions $\{f_n\}_{n\ge 1}$ so that
$$\mathbf{S}=\lim_{n\to\infty}\|D^{1/3}S(t)f_n\|_{L^6_{t,x}},\quad \|f_n\|_{L^2}=1.$$
Applying Theorem \ref{thm:linear-profile} to $f_n$: for any $l\ge 1$, there exists $\{\phi^j\}_{1\le j\le l}, w_n^l\in L^2$ and $(h_n^j,\xi_n^j,x_n^j,t_n^j)$ such that
$$
u_n=\sum_{1\le j\le l, \xi^j_n\equiv 0, \atop \text{ or } |h_n^j\xi_n^j|\to \infty} e^{it^j_n\Delta^2}g^j_n[e^{i(\cdot)h^j_n\xi^j_n}\phi^j]+w^l_n,$$
 where $$\lim_{l\to\infty}\lim_{n\to\infty} \|D^{1/3}S(t)w_n^l\|_{L^6_{t,x}}=0.$$
Combining it with the orthogonality results in Remark \ref{re:profiles-orthogonal}, we obtain
\begin{align*}
\mathbf{S}^6&=\lim_{n\to\infty}\|D^{1/3}S(t)f_n\|^6_{L^6_{t,x}}
=\lim_{l\to\infty}\lim_{n\to\infty}\|\sum_{1\le j\le l}D^{1/3}S(t+t_n^j)g_n^j[e^{ixh_n^j\xi_n^j}\phi^j]\|^6_{L^6_{t,x}}\\
&=\lim_{l\to\infty}\lim_{n\to\infty}\sum_{1\le j\le l}\|D^{1/3}S(t)[e^{ixh_n^j\xi_n^j}\phi^j]\|^6_{L^6_{t,x}}\le \mathbf{S}^6\sum_{j=1}^\infty\|\phi^j\|^6_{L^2}\\
&\le \mathbf{S}^6\left(\sum_{j=1}^\infty \|\phi^j\|^2_{L^2}\right)^3\le \mathbf{S}^6.
\end{align*}
Then all inequalities will become equal above. In particular, by the inclusion of $\ell^3$ into $\ell^1$,
we see that there is only $j$ remains and
$$\|\phi^j\|_{L^2}=1,\quad \mathbf{S}=\lim_{n\to\infty}\|D^{1/3}S(t)[e^{ixh_n^j\xi_n^j}\phi^j]\|_{L^6_{t,x}}.$$ So we consider the following two cases after fixing this $j$.
\begin{itemize}
\item If $\xi_n^j\equiv 0$, then $\phi^j$ is an extremiser as desired.

\item If $\lim_{n\to\infty}h_n^j\xi_n^j=\infty$, we set $a_n:=h_n^j\xi_n^j$; then
\begin{equation}\label{eq-72}
\mathbf{S}=\lim_{n\to\infty}\|D^{1/3}S(t)[e^{ixa_n}\phi^j]\|_{L^6_{t,x}},\,\|\phi^j\|_{L^2}=1.
\end{equation}
\end{itemize}
This establishes the first half of Theorem \ref{thm:dichotomy}. The following proposition will complete its proof.
\begin{proposition}\label{prop:convg-to-schr}
For any $\phi\in L^2$, we have the following convergence,
\begin{equation}\label{eq-convg-to-schr}
\lim_{N\to\infty}\|D^{1/3}S(t)[e^{ixN}\phi]\|_{L^6_{t,x}} =\|e^{-it\Delta} \phi\|_{L^6_{t,x}}.
\end{equation}
\end{proposition}
Let us postpone the proof of this proposition and continue the proof for Theorem \ref{thm:dichotomy}. On the one hand, by applying Proposition \ref{prop:convg-to-schr},
\begin{equation}\label{eq-73}
\mathbf{S}=\lim_{n\to\infty}\|D^{1/3}S(t)[e^{ixa_n}\phi^j]\|_{L^6_{t,x}}=\|e^{-it\Delta} \phi^j\|_{L^6_{t,x}}\le \mathbf{S}_{schr} \|\phi^j\|_{L^2}=\mathbf{S}_{schr}.
\end{equation}
On the other hand, by the works of Foschi \cite{Foschi:2007:maxi-strichartz-2d}, Hundertmark, Zharnitsky \cite{Hundertmark-Zharnitsky:2006:maximizers-Strichartz-low-dimensions} and Bennett Bez, Carbery, Hundertmark \cite{Bennett-Bez-Carbery-Hundertmark:2008:heat-flow-of-strichartz-norm}, we know that $\phi_0=e^{-|x|^2}$ is an extremal for $\mathbf{S}_{schr}$. Let $\phi=\phi_0$ in \eqref{eq-convg-to-schr}, we see that
\begin{equation}\label{eq-74}
\mathbf{S}_{schr}=\frac {\|e^{-it\Delta} \phi_0\|_{L^6_{t,x}}}{\|\phi_0\|_{L^2}}=\lim_{N\to \infty}\frac {\|D^{1/3}S(t)[e^{ixN}\phi_0]\|_{L^6_{t,x}}}{\|\phi_0\|_{L^2}}\le \mathbf{S}
\end{equation} by the definition of $\mathbf{S}$. Returning to \eqref{eq-73}, we see that
\begin{equation}
\mathbf{S}=\|e^{-it\Delta} \phi^j\|_{L^6_{t,x}}\le \mathbf{S}_{schr} \|\phi^j\|_{L^2}=\mathbf{S}_{schr}\le \mathbf{S}.
\end{equation}So this forces all inequality signs to be equal. In particular, we have
\begin{align}
\label{eq-75} \mathbf{S}&=\mathbf{S}_{schr},\\
\label{eq-76}  \|e^{-it\Delta} \phi^j\|_{L^6_{t,x}}&=\mathbf{S}_{schr} \|\phi^j\|_{L^2}.
\end{align}
In other words, \eqref{eq-76} says that $\phi^j$ is an extremal for the Strichartz inequality for the Schr\"odinger equation.  From the works of Foschi \cite{Foschi:2007:maxi-strichartz-2d}, Hundertmark, Zharnitsky \cite{Hundertmark-Zharnitsky:2006:maximizers-Strichartz-low-dimensions}, this information implies that $\phi^j$ is a Gaussian up to the natural symmetries associated to the Strichartz inequality for the Schr\"odinger equation. This finishes the proof of Theorem \ref{thm:dichotomy}.
\end{proof}

Now we present a proof for Proposition \ref{prop:convg-to-schr}.
\begin{proof}[Proof of Proposition \ref{prop:convg-to-schr}.] We may assume $\phi\in \mathcal{S}$ with compact Fourier support in $(-1,1)$, where $\mathcal{S}$ denotes the collection of Schwartz functions on $\mathbb{R}$. Then by a change of variables,
\begin{equation}\label{eq-77}
\begin{split}
\|D^{1/3}S(t)[e^{ixN}\phi]\|_{L^6_{t,x}}&=6^{1/6}\left\|\int e^{i\eta (x+4tN^3)+i\eta^2 6tN^2+i\eta^34tN+it\eta^4}|\eta+N|^{1/3} \widehat\phi(\eta)d\eta\right\|_{L^6_{t,x}}\\
&=\left\|\int e^{ix\eta+it\eta^2+it\eta^3\frac{2}{3N} +it\eta^4\frac {1}{6N^2}}|\frac \eta N+1|^{1/3} \widehat\phi(\eta)d\eta\right\|_{L^6_{t,x}}.
\end{split}
\end{equation}
Then the assertion in Proposition \ref{prop:convg-to-schr} reduces to
\begin{equation}\label{eq-78}
\lim_{N\to\infty}\left\|\int e^{ix\eta+it\eta^2+it\eta^3\frac{2}{3N} +it\eta^4\frac {1}{6N^2}}|\frac \eta N+1|^{1/3} \widehat\phi(\eta)d\eta\right\|_{L^6_{t,x}}=\|e^{-it\Delta} \phi\|_{L^6_{t,x}}.
\end{equation} This follows from the dominated convergence theorem.  Indeed, there holds that
\begin{equation*}
\int e^{ix\eta+it\eta^2+it\eta^3\frac{2}{3N} +it\eta^4\frac {1}{6N^2}}|\frac \eta N+1|^{1/3} \widehat\phi(\eta)d\eta
\to \int e^{ix\eta +it\eta^2 }\widehat{\phi} d\eta =e^{-it\Delta}\phi(x)
\end{equation*} for almost everywhere $(t,x)$ as $N$ goes to infinity. On the other hand, let
$$I(t,x):=\left|\int e^{ix\eta+it\eta^2+it\eta^3\frac{2}{3N} +it\eta^4\frac {1}{6N^2}}|\frac \eta N+1|^{1/3} \widehat\phi(\eta)d\eta \right|.$$ We aim to find a dominating function for $I(t,x)$. Since $|\eta|\le 1$, there exists $N_0>0$ such that
$$ 2\left|t(1+\eta\frac{2}{N} +\eta^2\frac {1}{N^2})\right|\ge c_0>0, \text{ for all } N\ge N_0,$$
where $c_0>0$ is an universal constant. Then the stationary phase estimate (see e.g., \cite[Chapter 8, p.334]{Stein:1993}) implies that, there always holds that
\begin{equation}\label{eq-79}
I(t,x)\le \frac {C_\phi}{(1+|t|)^{1/2}}
\end{equation}
for all $x\in \mathbb{R}_x$ and for all $N\ge N_0$. Fixing $t\in \mathbb{R}_t$, we split $\mathbb{R}_x$ into two parts,
\begin{align*}
\Omega(t)=\{x\in \mathbb{R}: \bigl||x|-3|t|\bigr|\le \frac {|x|}{2}\}, \text{ and } \mathbb{R}\setminus \Omega(t).
\end{align*}
On $\mathbb{R}\times \Omega(t)$, $|t|\sim |x|$ for $N\ge N_0$ and hence the dominating function can be chose as
$$F_1(t,x):= \frac {C_\phi}{(1+|t|)^{1/4}(1+|x|)^{1/4}}.$$
However on $\mathbb{R}\times \bigl(\mathbb{R}\setminus \Omega(t)\bigr)$, for each fixed $t$, we have
$$\bigl||x|-3|t|\bigr|\ge \frac {|x|}{2}.$$
Hence for all $ N\ge N_0$, the phase in $I(t,x)$ is non-stationary. This implies that
\begin{equation}\label{eq-80}
I(t,x)\le \frac {C_\phi}{1+|x|}.
\end{equation}
So on $\mathbb{R}\times \bigl(\mathbb{R}\setminus \Omega(t)\bigr)$, we combine the two upper bounds in \eqref{eq-79}, \eqref{eq-80}, and choose the dominating function to be
$$F_2(t,x)=\frac {C_\phi}{(1+|t|)^{1/4}(1+|x|)^{1/2}}.$$
Note that $F_1$ and $F_2$ are in $L^6_{t,x}$ for all $N\ge N_0$, which serve as dominating functions. Therefore we finish the proof of this proposition.
\end{proof}

\end{document}